\documentclass[11pt]{amsart}
\usepackage{amsmath,amssymb,mathrsfs}
\newtheorem{theorem}{{Theorem}}
\newtheorem*{theorem*}{Theorem}

\newtheorem{proposition}[theorem]{Proposition}

\newtheorem*{corollary*}{Corollary}
\newtheorem{conjecture}[theorem]{Conjecture}

\theoremstyle{definition}

\theoremstyle{remark}

\usepackage{lmodern,url,enumerate,mathtools,microtype,xcolor}
\usepackage{graphicx}
\usepackage{halloweenmath}

\theoremstyle{definition}
\newtheorem*{remark*}{Remark}
\newtheorem*{idea*}{Idea}
\newtheorem*{definition*}{Definition}
\newtheorem*{claim*}{Claim}

\newcommand{\mr}[1]{{\rm #1}}

\newcommand{\bR}{\mathbb{R}}

\newcommand{\nc}{\newcommand}

\nc{\on}{\operatorname}
\nc{\Spec}{\on{Spec}}
\nc{\dbar}{\overline{\partial}}
\nc{\osc}{\on{osc}}

\nc{\sn}{\mr{sn}}
\nc{\cn}{\mr{cn}}
\nc{\dn}{\mr{dn}}

\newcommand{\area}{\on{area}}

\title{On a Variant of the Penrose Conjecture}

\author[]{Sven Hirsch}
\address{Columbia University, 2990 Broadway, New York, NY 10027, USA}
\email{sven.hirsch@columbia.edu}

\author[]{Yipeng Wang}
\address{Columbia University, 2990 Broadway, New York, NY 10027, USA}
\email{yw3631@columbia.edu}

\begin{document}
\maketitle

\begin{abstract}
We give a counterexample to a recently conjectured variant of the Penrose inequality.
\end{abstract}

\medskip

In \cite[Section~4]{YauOpenProblems2024}, Yau proposed the following variant of the Penrose conjecture \cite{HuiskenIlmanen,Bray}.

\begin{conjecture}\label{conj:Yau}
Let $\alpha \ge 0$, and let $(M^3,g)$ be an asymptotically flat manifold with nonnegative scalar curvature that contains an outermost minimal surface $\Sigma$ with induced metric $\gamma$. Then
\begin{equation}\label{eqn:Yau-mass-eigenvalue-inequality}
m_{\on{ADM}}(M,g)\geq \sqrt{\frac{1}{2\lambda_1^{\alpha}(\Sigma,\gamma)}},
\end{equation}
where $\lambda_0^{\alpha}(\Sigma,\gamma)<\lambda_1^{\alpha}(\Sigma,\gamma)\leq \cdots$ denote the eigenvalues of the operator $-\Delta_{\Sigma}+\alpha K$, and $K$ denotes the Gauss curvature of $(\Sigma,\gamma)$.
\end{conjecture}

The motivation for Yau's conjecture is as follows. Taking $\alpha=0$ in Conjecture \ref{conj:Yau}, we have $\lambda_1^{\alpha}(\Sigma,\gamma)=\lambda_1(\Sigma,\gamma)$, the first nonzero eigenvalue of $(\Sigma,\gamma)$. By Hersch's theorem \cite{Hersch},
\[
\lambda_1(\Sigma,\gamma)\cdot \area(\Sigma,\gamma)\leq 8\pi,
\]
and therefore inequality \eqref{eqn:Yau-mass-eigenvalue-inequality} would imply that
\begin{equation}\label{eqn:penrose}
m_{\on{ADM}}(M,g)\geq \sqrt{\frac{\area(\Sigma,\gamma)}{16\pi}},
\end{equation}
which is precisely the classical Riemannian Penrose inequality.

\medskip

By a generalization of Hersch's theorem \cite[Theorem~2.2]{ElSoufiIlias1992}, for any $\alpha\in \bR$, we have
\[
\lambda_1^{\alpha}(\Sigma,\gamma)\cdot \area(\Sigma,\gamma)\leq (2+\alpha)4\pi.
\]
It therefore seems more natural to impose
\begin{equation}\label{eqn:alpha-mass-eigenvalue-inequality}
m_{\on{ADM}}(M,g)\geq \sqrt{\frac{2+\alpha}{4\lambda_1^{\alpha}(\Sigma,\gamma)}}
\end{equation}
in Conjecture \ref{conj:Yau}. In particular, \eqref{eqn:alpha-mass-eigenvalue-inequality} would imply the classical Riemannian Penrose inequality \eqref{eqn:penrose}.

\medskip

In this short note, we show that Conjecture \ref{conj:Yau} and inequality \eqref{eqn:alpha-mass-eigenvalue-inequality} are false in general.

\begin{theorem}\label{thm:main}
For each $L\ge 1$, there exists an asymptotically flat $3$-manifold $(M_L,g_L)$ with nonnegative scalar curvature that contains an outermost minimal surface $\Sigma_L$ such that $m_{\on{ADM}}(M_L,g_L)=1$. Moreover, for any $\alpha\ge 0$, $\lambda_1^\alpha(\Sigma_L)=O(L^{-1})$ as $L\to\infty$.
\end{theorem}

We use the following theorem of Mantoulidis--Schoen \cite[Theorem~2.1]{MantoulidisSchoen}.

\begin{theorem}\label{thm:MS}
Let $\gamma$ be a Riemannian metric on $S^2$ satisfying
\[
\lambda_0^1(S^2,\gamma)=\inf_{u\neq 0}\frac{\int_{S^2}|\nabla u|^2\,d\mu_{\gamma}+\int_{S^2}K u^2\,d\mu_{\gamma}}{\int_{S^2}u^2\,d\mu_{\gamma}}>0.
\]
Then for any $m>\sqrt{\frac{\area(S^2,\gamma)}{16\pi}}$, there exists an asymptotically flat manifold $(M^3,g)$ with nonnegative scalar curvature satisfying the following properties:
\begin{itemize}
    \item $\partial M$ is an outward-minimizing minimal surface isometric to $(S^2,\gamma)$.
    \item Outside a compact subset, $M$ is isometric to a Schwarzschild metric of mass $m$. In particular, $m_{\on{ADM}}(M,g)=m$.
    \item $M$ is foliated by mean convex spheres.
\end{itemize}
\end{theorem}

To complete the construction, we consider an oblong family of Riemannian metrics $\gamma_L$ on $S^2$.

\begin{proposition}
For every $L\ge 1$, there exists a Riemannian metric $\gamma_L$ on $S^2$ with positive Gauss curvature and area $4\pi$ such that, for every $\alpha\ge 0$, we have $\lambda_1^{\alpha}(S^2,\gamma_L)=O(L^{-1})$.
\end{proposition}

\begin{proof}
Fix $L\ge 1$ and consider the conformal metric
\[
\hat{\gamma}_L=e^{-2\psi_L}(dt\otimes dt+d\theta\otimes d\theta)
\]
on $\bR\times [0,2\pi]$, where
\[
\psi_L(t)=\log(1+e^{t-L})+\log(1+e^{-t-L}).
\]
The metric $\hat{\gamma}_L$ extends smoothly to $S^2$. Moreover,
\[
K_{\hat{\gamma}_L}(t,\theta)=e^{2\psi_L(t)}\psi_L''(t)=4e^{-2L}(1+\cosh L\cosh t)>0
\]
and
\[
\Delta_{\hat{\gamma}_L}u(t,\theta)=e^{2\psi_L}\left(\frac{\partial^2 u}{\partial t^2}+\frac{\partial^2 u}{\partial\theta^2}\right).
\]
We also have
\[
e^{-2\psi_L(t)}=\frac{e^{2L}}{4(\cosh t+\cosh L)^2},
\]
and hence
\begin{align*}
\area(S^2,\hat{\gamma}_L)
&=2\pi\int_{-\infty}^{\infty}e^{-2\psi_L(t)}\,dt\\
&=\frac{\pi e^{2L}}{2}\int_{-\infty}^{\infty}\frac{dt}{(\cosh t+\cosh L)^2}\\
&=4\pi\,\frac{L\coth L-1}{(1-e^{-2L})^2}\\
&=4\pi (L-1)+O(L e^{-2L}).
\end{align*}
Moreover, for each $\alpha\ge 0$, we have $\lambda_1^\alpha(S^2,\hat\gamma_L)=O(L^{-2})$ as $L\to\infty$. Indeed, since both $\hat\gamma_L$ and $K_{\hat\gamma_L}$ are even in $t$, the ground-state eigenfunction of $-\Delta_{\hat\gamma_L}+\alpha K_{\hat\gamma_L}$ is even in $t$. Therefore any odd function $u(t)$ is $L^2$-orthogonal to the ground state and is admissible in the Rayleigh characterization of $\lambda_1^\alpha$. Take
\[
u(t)=
\begin{cases}
\sin\!\big(\frac{\pi t}{L}\big),& |t|\le L,\\
0,& |t|\ge L.
\end{cases}
\]
Then
\[
\int_{S^2} |\nabla u|^2\,d\mu_{\hat\gamma_L}=O(L^{-1}),\qquad
\int_{S^2} K_{\hat\gamma_L}u^2\,d\mu_{\hat\gamma_L}=O(L^{-2}),
\]
and
\[
\int_{S^2} u^2\,d\mu_{\hat\gamma_L}\ge cL.
\]
The Rayleigh quotient therefore yields $\lambda_1^\alpha(S^2,\hat\gamma_L)=O(L^{-2})$ as $L\to\infty$.

\medskip

We now define a normalized metric on $S^2$ by
\[
\gamma_L=\frac{4\pi}{\area(S^2,\hat{\gamma}_L)}\hat{\gamma}_L.
\]
It follows that $\area(S^2,\gamma_L)=4\pi$ and $\lambda_1^{\alpha}(S^2,\gamma_L)=O(L^{-1})$ as $L\to\infty$.
\end{proof}

\begin{proof}[Proof of Theorem \ref{thm:main}]
By Theorem \ref{thm:MS}, we can find a family of asymptotically flat manifolds $(M_L^3,g_L)$ with nonnegative scalar curvature containing a horizon boundary $\Sigma_L$. Moreover, we may arrange that $m_{\on{ADM}}(M_L,g_L)=1$ and $\lambda^{\alpha}_1(\Sigma_L)=O(L^{-1})$ as $L\to\infty$ for every $\alpha\ge 0$.
\end{proof}

Theorem \ref{thm:MS} essentially shows that any Penrose-type inequality must either be formulated in terms of $\area(\Sigma)$ or depend on more than the intrinsic geometry of $\Sigma$, as in the mass-$p$-capacity inequalities \cite{Bray,AgostinianiMazzieriOronzio}.

\medskip

\noindent \textbf{Acknowledgments.} The authors would like to thank S.-T.~Yau for helpful discussions and for his interest in this work.


\nocite{*}
\begin{thebibliography}{10}

\bibitem{AgostinianiMazzieriOronzio}
V.~Agostiniani, C.~Mantegazza, L.~Mazzieri, and F.~Oronzio,
\newblock Riemannian Penrose inequality via nonlinear potential theory,
\newblock {\em Ann. Sc. Norm. Super. Pisa Cl. Sci.} \textbf{27} (2025), doi:10.2422/2036-2145.202409\_028.

\bibitem{Bray}
H.~L.~Bray,
\newblock Proof of the Riemannian Penrose inequality using the positive mass theorem,
\newblock {\em J. Differential Geom.} \textbf{59} (2001), no.~2, 177--267.

\bibitem{ElSoufiIlias1992}
A.~El~Soufi and S.~Ilias,
\newblock Majoration de la seconde valeur propre d'un op{\'e}rateur de Schr{\"o}dinger sur une vari{\'e}t{\'e} compacte et applications,
\newblock {\em J. Funct. Anal.} \textbf{103} (1992), no.~2, 294--316.

\bibitem{Hersch}
J.~Hersch,
\newblock Quatre propri\'et\'es isop\'erim\'etriques de membranes sph\'eriques homog\`enes,
\newblock {\em C. R. Acad. Sci. Paris S\'er. A-B} \textbf{270} (1970), A1645--A1648.

\bibitem{HuiskenIlmanen}
G.~Huisken and T.~Ilmanen,
\newblock The inverse mean curvature flow and the Riemannian Penrose inequality,
\newblock {\em J. Differential Geom.} \textbf{59} (2001), no.~3, 353--437.

\bibitem{MantoulidisSchoen}
C.~Mantoulidis and R.~Schoen,
\newblock On the Bartnik mass of apparent horizons,
\newblock {\em Classical Quantum Gravity} \textbf{32} (2015), no.~20, 205002.

\bibitem{YauOpenProblems2024}
S.-T.~Yau,
\newblock Open problems in geometry,
\newblock {\em ICCM Not. Notices Int. Consort. Chinese Math.} \textbf{12} (2024), no.~1, 35--46.

\end{thebibliography}
\end{document}